\documentclass[11pt]{amsart}

\usepackage{latexsym}
\usepackage{amsmath}
\usepackage{amsfonts}
\usepackage{amssymb}
\usepackage{enumitem}
\usepackage {tikz}
\usepackage{color}
\usetikzlibrary {positioning}
\usetikzlibrary{arrows,shapes}
\usepackage{color}
\usepackage{graphicx}
\usepackage{latexsym}
\usepackage{soul}

\usepackage[noend]{algorithmic} 
\usepackage{algorithm,caption}
\algsetup{indent=2em} 
 
\usepackage{tikz}
\usepackage{color}
\usepackage{placeins}
\usetikzlibrary{positioning}
\usetikzlibrary{arrows,shapes}

\newtheorem{theorem}{Theorem}[section]
\newtheorem{lemma}[theorem]{Lemma}

\newtheorem{proposition}[theorem]{Proposition}

\theoremstyle{definition}

\theoremstyle{remark}

\numberwithin{equation}{section}

\newcommand\Ga{\mathcal{G}^\mathrm{apex}}

\begin{document}

\title[Apex Graphs and Cographs]{Apex Graphs and Cographs}

\author[Singh]{Jagdeep Singh}
\address{Department of Mathematics and Statistics\\
Binghamton University\\
Binghamton, New York}
\email{singhjagdeep070@gmail.com}

\author[Sivaraman]{Vaidy Sivaraman}
\address{Department of Mathematics and Statistics\\
Mississippi State University\\
Mississippi}
\email{vaidysivaraman@gmail.com}

\author[Zaslavsky]{Thomas Zaslavsky}
\address{Department of Mathematics and Statistics\\
Binghamton University\\
Binghamton, New York}
\email{zaslav@math.binghamton.edu}


\date{\today}

\begin{abstract}
A class $\mathcal{G}$ of graphs is called hereditary if it is closed under taking induced subgraphs. We denote by $\Ga$ the class of graphs $G$ that contain a vertex $v$ such that $G-v$ is in $\mathcal{G}$. Borowiecki, Drgas-Burchardt, and Sidorowicz proved that if a hereditary class $\mathcal{G}$ has finitely many forbidden induced subgraphs, then so does $\Ga$. We provide an elementary proof of this result.  

The hereditary class of cographs consists of all graphs $G$ that can be generated from $K_1$ using complementation and disjoint union. A graph is an apex cograph if it contains a vertex whose deletion results in a cograph. Cographs are precisely the graphs that do not have the $4$-vertex path as an induced subgraph. Our main result finds all such forbidden induced subgraphs for the class of apex cographs.
\end{abstract}

\maketitle

\section{Introduction}
\label{intro}
We consider finite, simple, and undirected graphs. An {\bf induced subgraph} of a graph $G$ is a graph $H$ that can be obtained from $G$ by a sequence of vertex deletions; we write $G[T]$ for the subgraph if $T$ is the set of remaining vertices. We say that $G$ {\it contains} $H$ when $H$ is an induced subgraph of $G$. A class $\mathcal{G}$ of graphs is called {\bf hereditary} if it is closed under taking induced subgraphs. 

A {\bf cograph} is a graph that can be generated from the single-vertex graph $K_1$ using the operations of complementation and disjoint union. 
Due to the following characterization, cographs are also called $P_4$-free graphs \cite{corneil}. 

\begin{theorem}
\label{cographs_characterisation}
A graph $G$ is a cograph if and only if $G$ does not contain the path $P_4$ on four vertices as an induced subgraph.
\end{theorem}

The {\bf apex class, $\Ga$,} for a class $\mathcal{G}$ is the class of graphs $G$ that contain a vertex $v$ such that $G-v$ is in $\mathcal{G}$. Note that if a graph class $\mathcal{G}$ is hereditary, then $\Ga$ is also hereditary. {A graph $G$ is an {\bf apex cograph} if it has a vertex $v$ such that $G-v$ is a cograph.}  {In Theorem \ref{cograph_apex} we characterize apex cographs in the spirit of Theorem \ref{cographs_characterisation}.}

For a hereditary class $\mathcal{G}$, we call a graph $H$ a {\bf forbidden induced subgraph} for $\mathcal{G}$ if $H$ is not in $\mathcal{G}$ but every proper induced subgraph of $H$ is in $\mathcal{G}$. Thus, a graph $G$ is in $\mathcal{G}$ if and only if $G$ contains no forbidden induced subgraph for $\mathcal{G}$. It is easy to see that the class of cographs is hereditary; by Theorem \ref{cographs_characterisation} it has only one forbidden induced subgraph. Borowiecki, Drgas-Burchardt and Sidorowicz \cite{boro} proved that in general, if a hereditary class $\mathcal{G}$ has a finite number of forbidden induced subgraphs, then so does $\Ga$. The precise statement is the following theorem.
 
\begin{theorem}
\label{main_apex}
Let $\mathcal{G}$ be a hereditary class of graphs and let $c$ denote the maximum number of vertices in a forbidden induced subgraph for $\mathcal{G}$. Then every forbidden induced subgraph for $\Ga$ has at most $\lfloor{\frac14(c+2)^2} \rfloor$ vertices.
\end{theorem}

We give a new, simple proof of Theorem \ref{main_apex} in Section 2.  As a consequence of Theorem \ref{main_apex}, the number of vertices of a forbidden induced subgraph for the class of apex cographs is at most nine.  {In Section 3, we show that there are none with nine vertices and we list all of the forbidden induced subgraphs.}

\begin{theorem}
 \label{cograph_apex}
 Let $G$ be a forbidden induced subgraph for the class of apex cographs. Then $5 \leq |V(G)| \leq 8$, and $G$ or its complement $\overline{G}$ is isomorphic to $C_5$ or one of the graphs in Figures \ref{six}, \ref{seven}, and \ref{eight}.
\end{theorem}

\section{Apex Exclusions}

The main part of the proof of Theorem \ref{main_apex} is the following reduction lemma.

\begin{proposition}
\label{bound_k}
Let $\mathcal{G}$ be a hereditary class of graphs and let $c$ denote the maximum number of vertices in a forbidden induced subgraph for $\mathcal{G}$. 
If $G$ is a forbidden induced subgraph for $\Ga$, then there are distinct subsets $S$ and $T$  of $V(G)$  such that $G[S]$ and $G[T]$ induce a forbidden induced subgraph for $\mathcal{G}$. Moreover, if $|S \cap T| = k$ is minimal, then $|V(G)| \leq (c-k)(k+2)$.
\end{proposition}

\begin{proof}
Since $G$ is not in $\Ga$, there is no vertex $v$ of $G$ such that $G-v$ is in $\mathcal{G}$.  Thus, $G-v$ contains a forbidden induced subgraph for $\mathcal{G}$ for every vertex $v$ of $G$. It follows that   
$V(G)$ has distinct subsets $S$ and $T$ such that each of $G[S]$ and $G[T]$ induces a forbidden induced subgraph for $\mathcal{G}$. 

First we consider the possibility that $S \cap T$ is empty.  If it were empty, there would be no vertex in $V(G)-(S \cup T)$, because if there were such a vertex $v$, then $G-v$ would be a proper induced subgraph of $G$ but not in $\Ga$ as it would contain two vertex-disjoint forbidden induced subgraphs for $\mathcal{G}$. This is a contradiction to the minimality of $G$. It follows that $V(G) = S \cup T$, so $|V(G)| \leq 2c$, verifying the proposition. Therefore, we may assume $S\cap T$ is nonempty for every choice of $S$ and $T$.

We pick $S$ and $T$ such that $|S \cap T|$ is minimal and let $S \cap T = \{v_1, \ldots, v_k\}$. Since for each $1 \leq i \leq k$, the graph $G-v_i$ is not in $\mathcal{G}$, we have a subset $F_i$ of $V(G)-v_i$ such that $G[F_i]$ induces a forbidden induced subgraph for $\mathcal{G}$. Observe that, if there is a vertex $x$ in $V(G)- (S \cup T \cup \bigcup_{i=1}^{k} F_i )$, then $G-x$ is not in $\Ga$, a contradiction. Therefore, $V(G) = S \cup T \cup \bigcup_{i=1}^{k} F_i$.

Next we show that for every $1 \leq i \leq k$, we have $|F_i \cap (S \cup T)| \geq k+1$. By our choice of $S$ and $T$, we have $|F_i \cap S| \geq k$ and $|F_i \cap T| \geq k$. Since $F_i$ does not contain $v_i$, it follows that $F_i$ intersects both $S-T$ and $T-S$. Therefore $|F_i \cap (S \cup T)| \geq k+1$. Since $V(G) = S \cup T \cup \bigcup_{i=1}^{k} F_i$, we have $|V(G)| \leq c + (c-k) + k(c-(k+1)) = (c-k)(k+2)$.
\end{proof}

\begin{proof}[Proof of Theorem \ref{main_apex}]
By Proposition \ref{bound_k}, we have $|V(G)| \leq (c-k)(k+2)$ where $0 \leq k \leq c$. The upper bound attains its maximum value when $k = \frac{c-2}{2}$. Therefore, $|V(G)| \leq \frac14(c+2)^2$. 
\end{proof}

\section{Apex Cograph Exclusions}

In this section, the first lemma  observes that the forbidden induced subgraphs for apex cographs are closed under complementation. We omit its straightforward proof.

\begin{lemma}
\label{closed_complement}
If $G$ is a forbidden induced subgraph for the class of apex cographs, then so is its complement, $\overline{G}$.
\end{lemma}

As for any given graph, the graph or its complement is connected, by Lemma \ref{closed_complement} we need to consider only connected graphs to find all forbidden induced subgraphs for apex cographs up to complementation.  Such forbidden graphs have at least five vertices, because any graph with fewer vertices is an apex cograph. We implemented the following algorithm on connected graphs using SageMath~\cite{sage}; in Figures \ref{six}--\ref{eight} we  list all such forbidden induced subgraphs up to complementation. The graphs were drawn using SageMath.

\begin{algorithm}
\renewcommand{\thealgorithm}{}
\label{pseudocode2}
\caption{Finding forbidden induced subgraphs for apex cographs up to complementation}

\begin{algorithmic}
\REQUIRE $n = 5,6,7,8,$ or $9$.
\STATE Set PreList $\leftarrow \emptyset$, FinalList $\leftarrow \emptyset$
\STATE Generate all connected graphs of order $n$ using nauty geng \cite{nauty} and store in an iterator $L$

\FOR{$g$ in $L$ such that $g$ is not a cograph}
 \STATE   Set $i \leftarrow 0$
    \FOR{$v$ in $V(g)$}
        \STATE $h = g - v$
        \IF{$h$ is not a cograph and $h-w$ is a cograph for some $w$ in $V(h)$}
            \STATE $i \leftarrow i+1$
        \ENDIF
    \ENDFOR

\IF{$i$ equals $|V(g)|$}

\STATE Add $g$ to PreList

\ENDIF

\ENDFOR

\FOR {$g$ in PreList}
    \IF{FinalList does not contain $\overline{g}$}
     \STATE add $g$ to FinalList
    
    \ENDIF
\ENDFOR

\end{algorithmic}
\end{algorithm}

The results were the following:

{\bf Graphs on five vertices.} The $5$-vertex cycle $C_5$ is the unique $5$-vertex forbidden induced subgraph for apex cographs.

{\bf Graphs on six vertices.} There are eight forbidden induced subgraphs for apex cographs on six vertices, namely, the $6$-vertex graphs shown in Figure~\ref{six} and their complements.

{\bf Graphs on seven vertices.} There are $62$ forbidden induced subgraphs for apex cographs on seven vertices, the graphs in Figure~\ref{seven} and their complements.

{\bf Graphs on eight vertices.} There are $75$ forbidden induced subgraphs for apex cographs on eight vertices, the graphs in Figure~\ref{eight} and their complements. Note that three out of these $75$ graphs are self-complementary.

\begin{figure}[htbp]
\centering
\begin{minipage}{.20\linewidth}
  \includegraphics[scale=0.30]{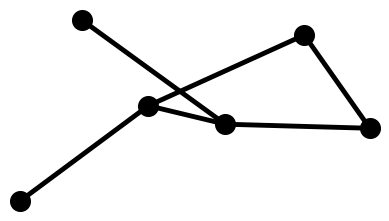}
\end{minipage} \hspace{.07\linewidth}
\begin{minipage}{.20\linewidth}
  \includegraphics[scale=0.25]{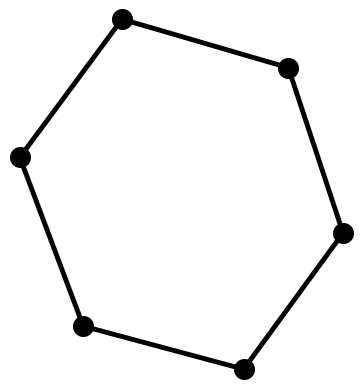}
\end{minipage} \hspace{.03\linewidth}
\begin{minipage}{.20\linewidth}
\includegraphics[scale=0.30]{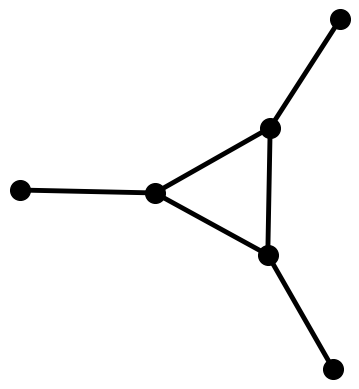}
\end{minipage}\hspace{.03\linewidth}
\begin{minipage}{.20\linewidth}
  \includegraphics[scale=0.25]{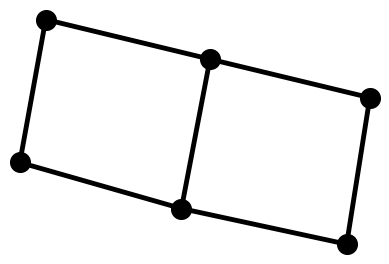}
\end{minipage}

\caption{These and their complements are the $6$-vertex forbidden induced subgraphs for apex cographs.}
\label{six}

\end{figure}

\begin{figure}[htbp]
\centering

\begin{minipage}{.17\linewidth}
  \includegraphics[scale=0.22]{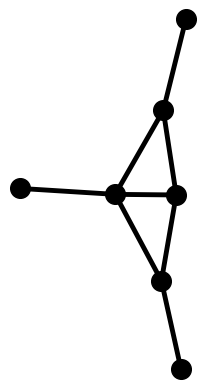}
\end{minipage} \hspace{.01\linewidth}
\begin{minipage}{.17\linewidth}
  \includegraphics[scale=0.22]{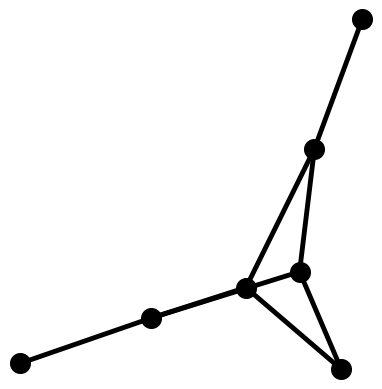}
\end{minipage}
\hspace{.01\linewidth}
\begin{minipage}{.17\linewidth}
\includegraphics[scale=0.22]{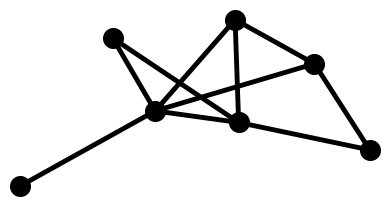}
\end{minipage}
\hspace{.01\linewidth}
\begin{minipage}{.17\linewidth}
\includegraphics[scale=0.22]{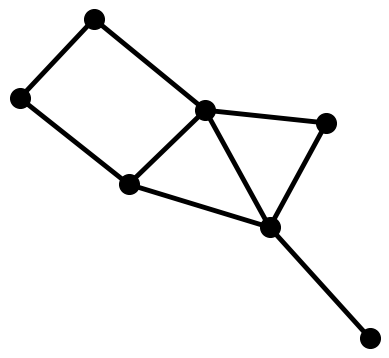}
\end{minipage}
\begin{minipage}{.17\linewidth}
  \includegraphics[scale=0.22]{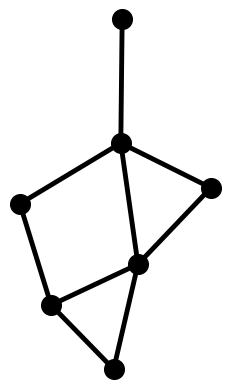}
\end{minipage}
\hspace{.01\linewidth}

\begin{minipage}{.17\linewidth}
  \includegraphics[scale=0.22]{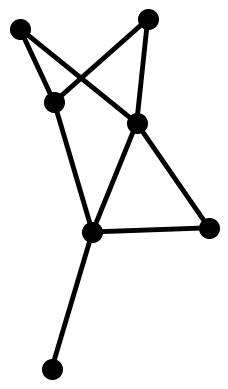}
\end{minipage} \hspace{.01\linewidth}
\begin{minipage}{.17\linewidth}
  \includegraphics[scale=0.22]{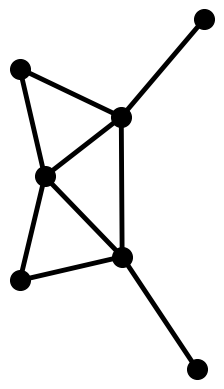}
\end{minipage}
\hspace{.01\linewidth}
\begin{minipage}{.17\linewidth}
\includegraphics[scale=0.22]{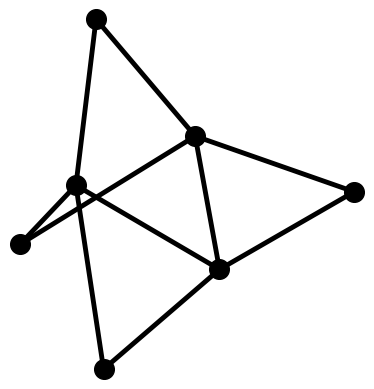}
\end{minipage}
\hspace{.01\linewidth}
\begin{minipage}{.17\linewidth}
\includegraphics[scale=0.22]{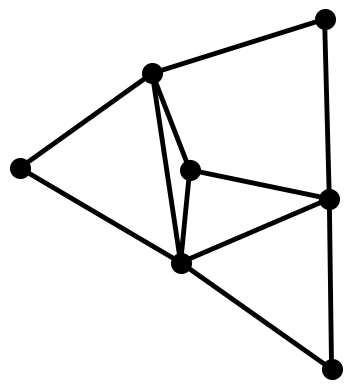}
\end{minipage}
\begin{minipage}{.17\linewidth}
  \includegraphics[scale=0.22]{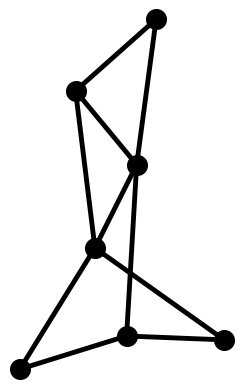}
\end{minipage}
\hspace{.01\linewidth}

\begin{minipage}{.17\linewidth}
  \includegraphics[scale=0.22]{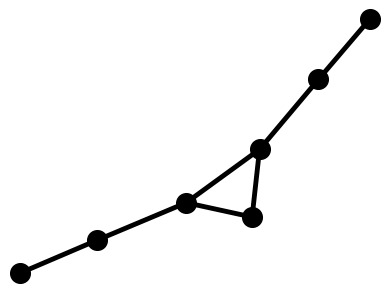}
\end{minipage} \hspace{.01\linewidth}
\begin{minipage}{.17\linewidth}
  \includegraphics[scale=0.22]{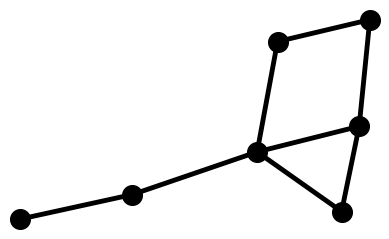}
\end{minipage}
\hspace{.01\linewidth}
\begin{minipage}{.17\linewidth}
\includegraphics[scale=0.22]{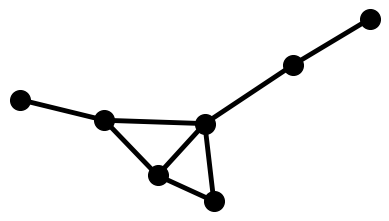}
\end{minipage}
\hspace{.01\linewidth}
\begin{minipage}{.17\linewidth}
\includegraphics[scale=0.22]{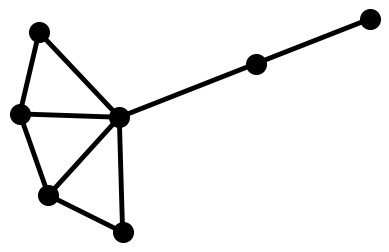}
\end{minipage}
\begin{minipage}{.17\linewidth}
  \includegraphics[scale=0.22]{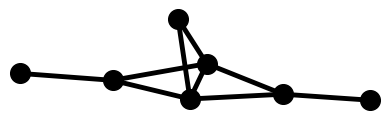}
\end{minipage}
\hspace{.01\linewidth}

\begin{minipage}{.17\linewidth}
  \includegraphics[scale=0.22]{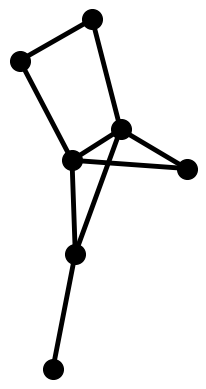}
\end{minipage} \hspace{.01\linewidth}
\begin{minipage}{.17\linewidth}
  \includegraphics[scale=0.22]{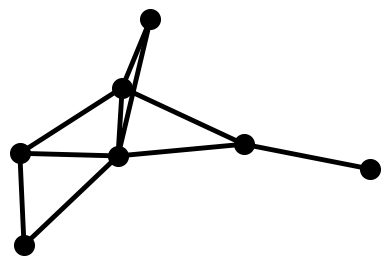}
\end{minipage}
\hspace{.01\linewidth}
\begin{minipage}{.17\linewidth}
\includegraphics[scale=0.22]{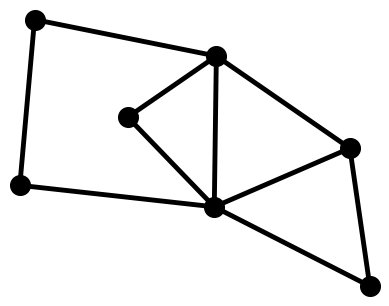}
\end{minipage}
\hspace{.01\linewidth}
\begin{minipage}{.17\linewidth}
\includegraphics[scale=0.22]{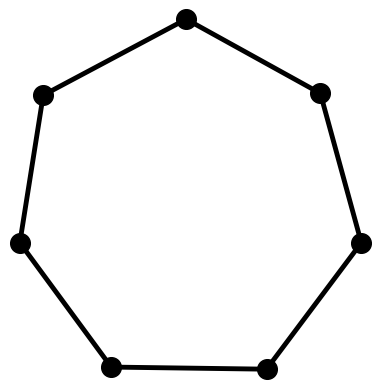}
\end{minipage}
\begin{minipage}{.17\linewidth}
  \includegraphics[scale=0.22]{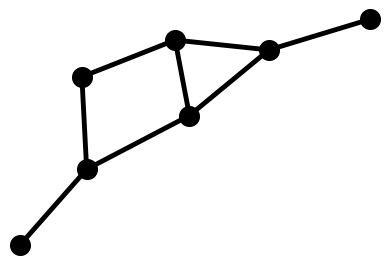}
\end{minipage}
\hspace{.01\linewidth}

\begin{minipage}{.17\linewidth}
  \includegraphics[scale=0.22]{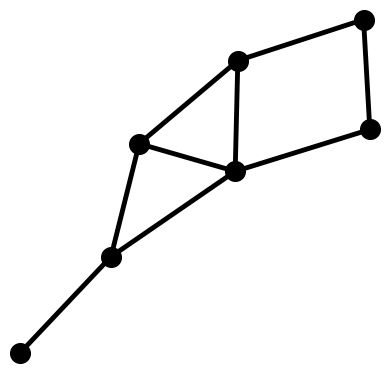}
\end{minipage} \hspace{.01\linewidth}
\begin{minipage}{.17\linewidth}
  \includegraphics[scale=0.22]{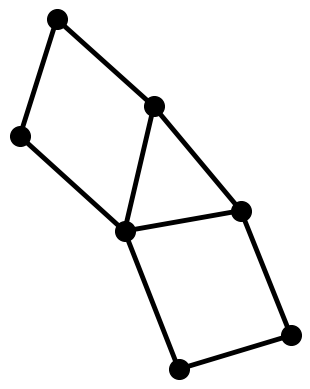}
\end{minipage}
\hspace{.01\linewidth}
\begin{minipage}{.17\linewidth}
\includegraphics[scale=0.22]{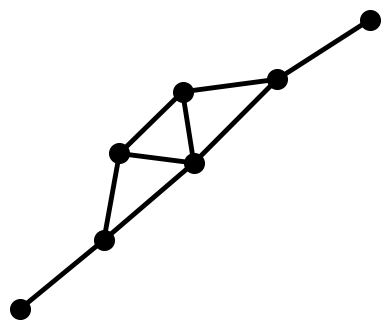}
\end{minipage}
\hspace{.01\linewidth}
\begin{minipage}{.17\linewidth}
\includegraphics[scale=0.22]{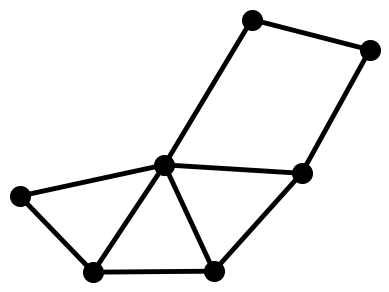}
\end{minipage}
\begin{minipage}{.17\linewidth}
  \includegraphics[scale=0.22]{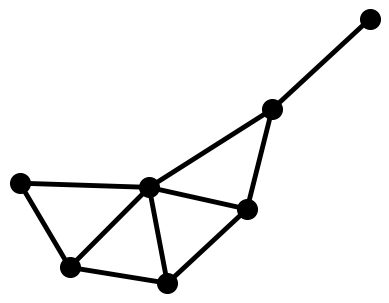}
\end{minipage}
\hspace{.01\linewidth}

\begin{minipage}{.17\linewidth}
  \includegraphics[scale=0.22]{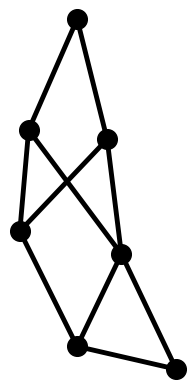}
\end{minipage} \hspace{.01\linewidth}
\begin{minipage}{.17\linewidth}
  \includegraphics[scale=0.22]{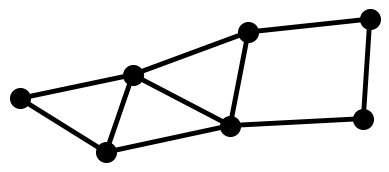}
\end{minipage}
\hspace{.01\linewidth}
\begin{minipage}{.17\linewidth}
\includegraphics[scale=0.22]{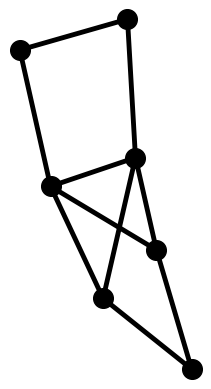}
\end{minipage}
\hspace{.01\linewidth}
\begin{minipage}{.17\linewidth}
\includegraphics[scale=0.22]{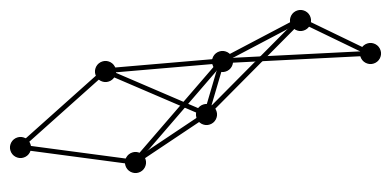}
\end{minipage}
\begin{minipage}{.17\linewidth}
  \includegraphics[scale=0.22]{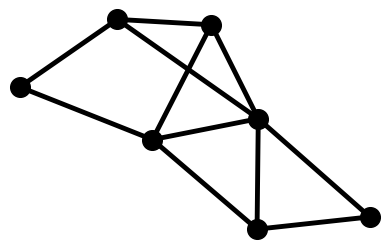}
\end{minipage}
\hspace{.01\linewidth}

\begin{minipage}{.17\linewidth}
  \includegraphics[scale=0.22]{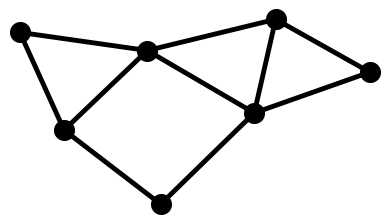}
\end{minipage} 

\caption{These and their complements are the $7$-vertex forbidden induced subgraphs for apex cographs.}

\label{seven}

\end{figure}

\begin{figure}[htbp]
\centering

\begin{minipage}{.17\linewidth}
  \includegraphics[scale=0.22]{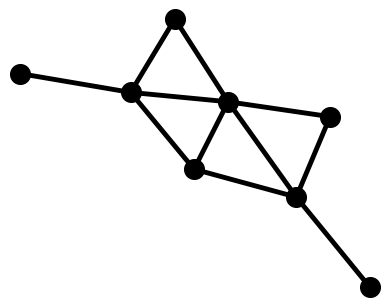}
\end{minipage} \hspace{.01\linewidth}
\begin{minipage}{.17\linewidth}
  \includegraphics[scale=0.22]{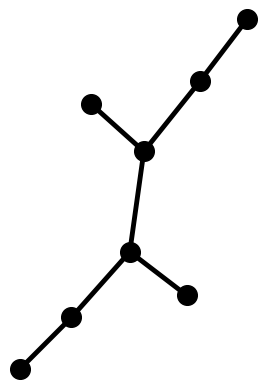}
\end{minipage}
\hspace{.01\linewidth}
\begin{minipage}{.17\linewidth}
\includegraphics[scale=0.22]{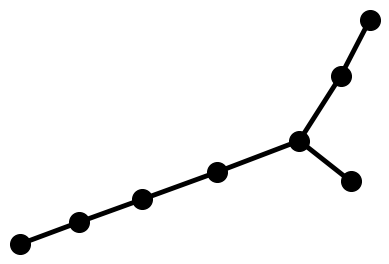}
\end{minipage}
\hspace{.01\linewidth}
\begin{minipage}{.17\linewidth}
\includegraphics[scale=0.22]{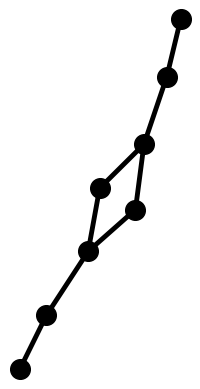}
\end{minipage}
\begin{minipage}{.17\linewidth}
  \includegraphics[scale=0.22]{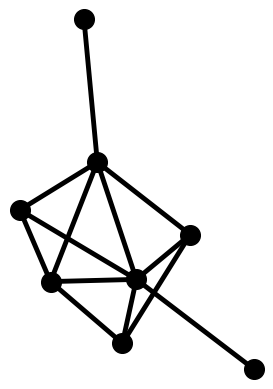}
\end{minipage}
\hspace{.01\linewidth}

\begin{minipage}{.17\linewidth}
  \includegraphics[scale=0.22]{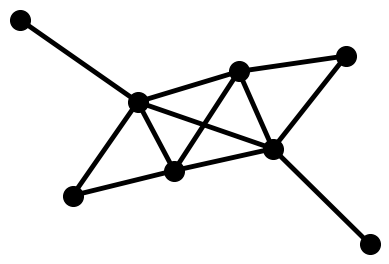}
\end{minipage} \hspace{.01\linewidth}
\begin{minipage}{.17\linewidth}
  \includegraphics[scale=0.22]{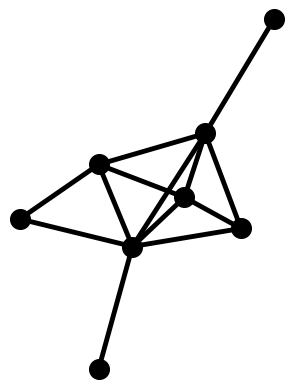}
\end{minipage}
\hspace{.01\linewidth}
\begin{minipage}{.17\linewidth}
\includegraphics[scale=0.22]{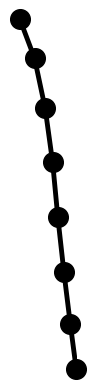}
\end{minipage}
\hspace{.01\linewidth}
\begin{minipage}{.17\linewidth}
\includegraphics[scale=0.22]{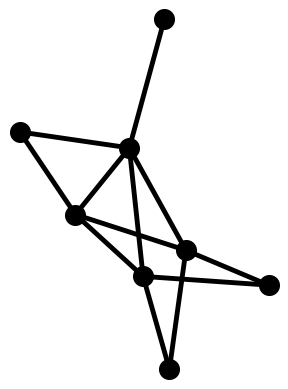}
\end{minipage}
\begin{minipage}{.17\linewidth}
  \includegraphics[scale=0.22]{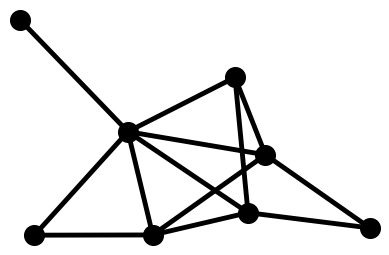}
\end{minipage}
\hspace{.01\linewidth}

\begin{minipage}{.17\linewidth}
  \includegraphics[scale=0.22]{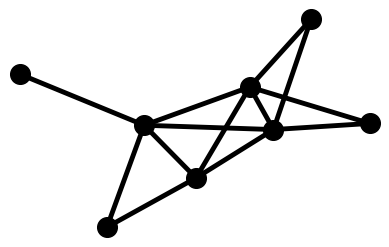}
\end{minipage} \hspace{.01\linewidth}
\begin{minipage}{.17\linewidth}
  \includegraphics[scale=0.22]{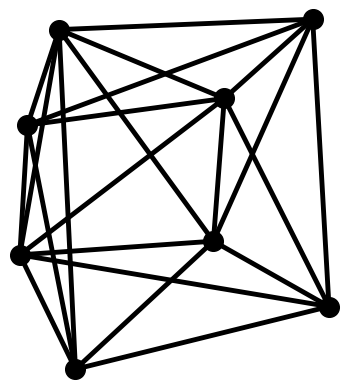}
\end{minipage}
\hspace{.01\linewidth}
\begin{minipage}{.17\linewidth}
\includegraphics[scale=0.22]{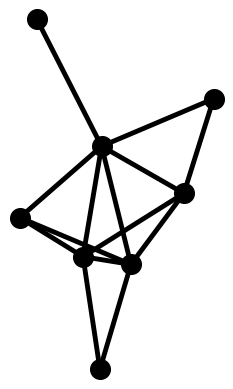}
\end{minipage}
\hspace{.01\linewidth}
\begin{minipage}{.17\linewidth}
\includegraphics[scale=0.22]{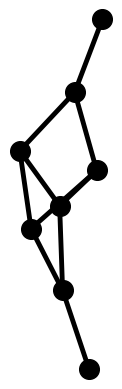}
\end{minipage}
\begin{minipage}{.17\linewidth}
  \includegraphics[scale=0.22]{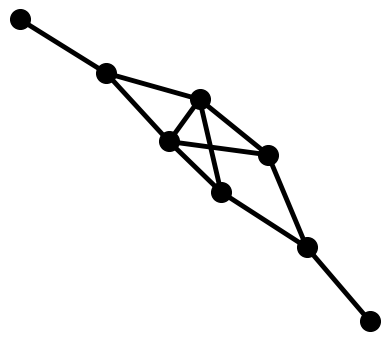}
\end{minipage}
\hspace{.01\linewidth}

\begin{minipage}{.17\linewidth}
  \includegraphics[scale=0.22]{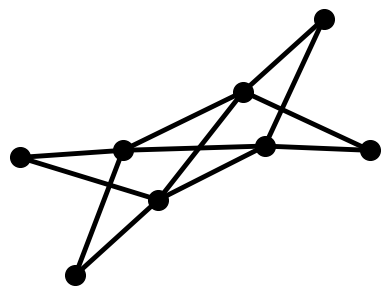}
\end{minipage} \hspace{.01\linewidth}
\begin{minipage}{.17\linewidth}
  \includegraphics[scale=0.22]{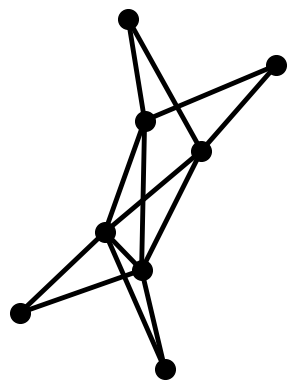}
\end{minipage}
\hspace{.01\linewidth}
\begin{minipage}{.17\linewidth}
\includegraphics[scale=0.22]{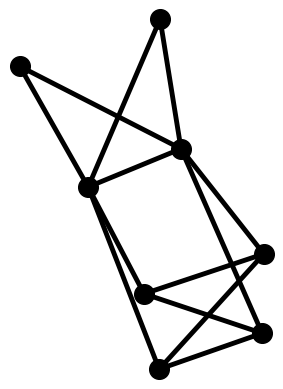}
\end{minipage}
\hspace{.01\linewidth}
\begin{minipage}{.17\linewidth}
\includegraphics[scale=0.22]{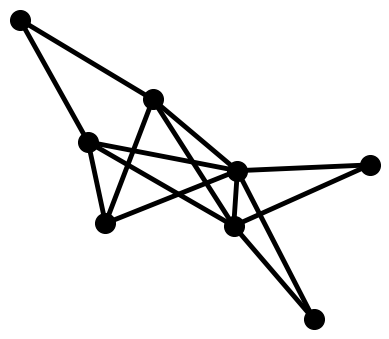}
\end{minipage}
\begin{minipage}{.17\linewidth}
  \includegraphics[scale=0.22]{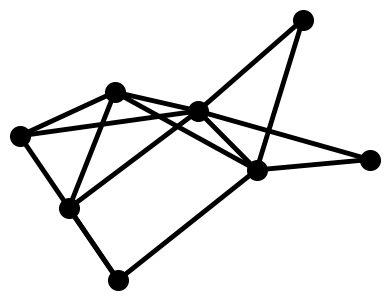}
\end{minipage}
\hspace{.01\linewidth}

\begin{minipage}{.17\linewidth}
  \includegraphics[scale=0.22]{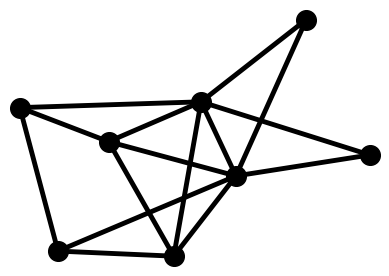}
\end{minipage} \hspace{.01\linewidth}
\begin{minipage}{.17\linewidth}
  \includegraphics[scale=0.22]{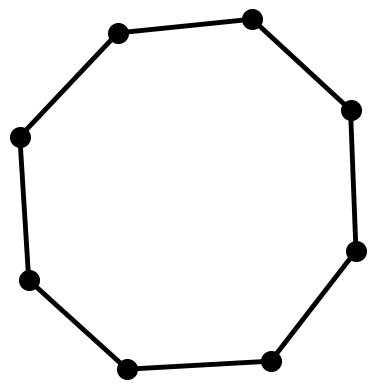}
\end{minipage}
\hspace{.01\linewidth}
\begin{minipage}{.17\linewidth}
\includegraphics[scale=0.22]{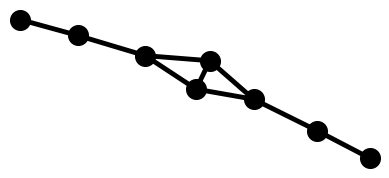}
\end{minipage}
\hspace{.01\linewidth}
\begin{minipage}{.17\linewidth}
\includegraphics[scale=0.22]{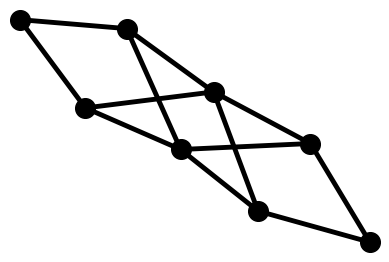}
\end{minipage}
\begin{minipage}{.17\linewidth}
  \includegraphics[scale=0.22]{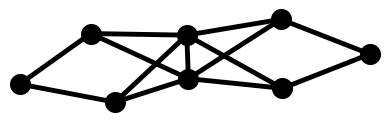}
\end{minipage}
\hspace{.01\linewidth}

\begin{minipage}{.17\linewidth}
  \includegraphics[scale=0.22]{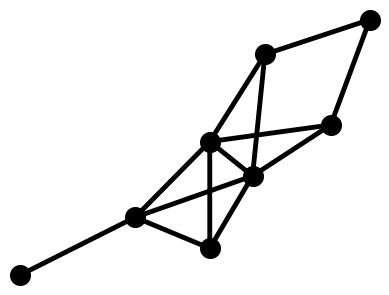}
\end{minipage} \hspace{.01\linewidth}
\begin{minipage}{.17\linewidth}
  \includegraphics[scale=0.22]{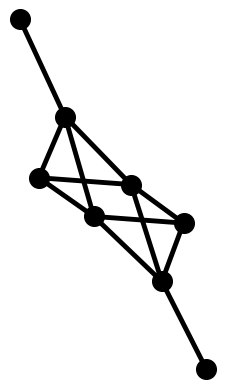}
\end{minipage}
\hspace{.01\linewidth}
\begin{minipage}{.17\linewidth}
\includegraphics[scale=0.22]{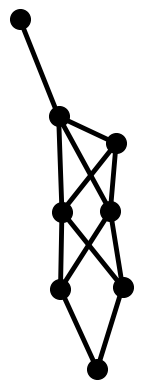}
\end{minipage}
\hspace{.01\linewidth}
\begin{minipage}{.17\linewidth}
\includegraphics[scale=0.22]{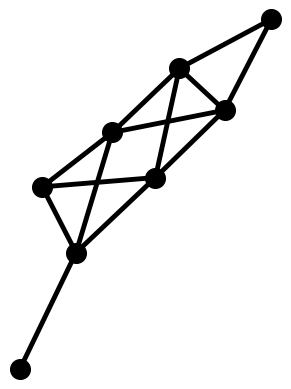}
\end{minipage}
\begin{minipage}{.17\linewidth}
  \includegraphics[scale=0.22]{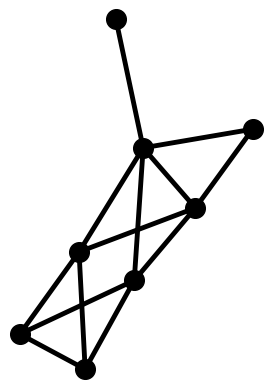}
\end{minipage}
\hspace{.01\linewidth}

\begin{minipage}{.17\linewidth}
  \includegraphics[scale=0.22]{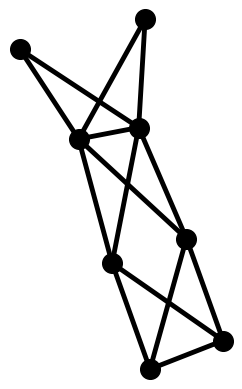}
\end{minipage} \hspace{.01\linewidth}
\begin{minipage}{.17\linewidth}
  \includegraphics[scale=0.22]{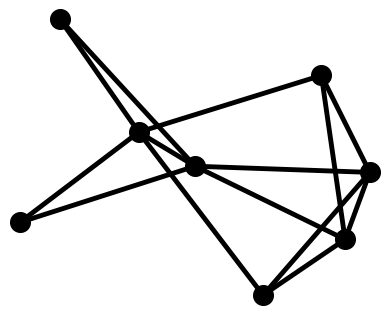}
\end{minipage}
\hspace{.01\linewidth}
\begin{minipage}{.17\linewidth}
\includegraphics[scale=0.22]{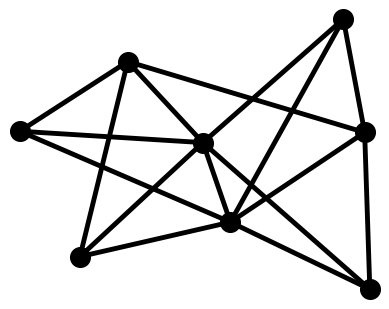}
\end{minipage}
\hspace{.01\linewidth}
\begin{minipage}{.17\linewidth}
\includegraphics[scale=0.22]{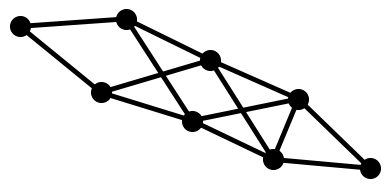}
\end{minipage}
\begin{minipage}{.17\linewidth}
  \includegraphics[scale=0.22]{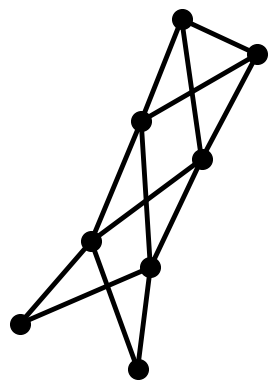}
\end{minipage}
\hspace{.01\linewidth}

\begin{minipage}{.17\linewidth}
  \includegraphics[scale=0.22]{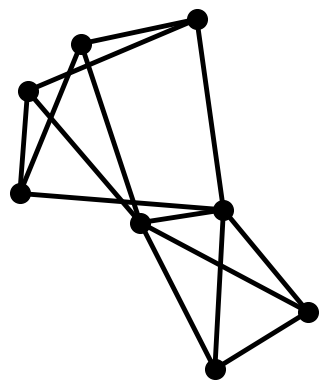}
\end{minipage} \hspace{.01\linewidth}
\begin{minipage}{.17\linewidth}
  \includegraphics[scale=0.22]{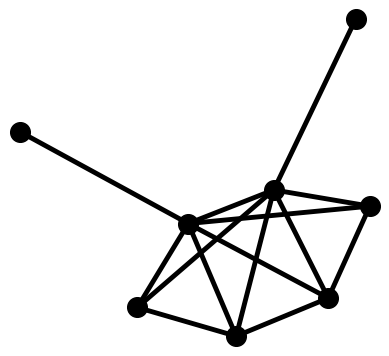}
\end{minipage}
\hspace{.01\linewidth}
\begin{minipage}{.17\linewidth}
\includegraphics[scale=0.22]{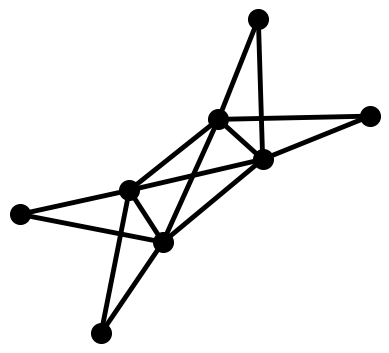}
\end{minipage}
\hspace{.01\linewidth}
\begin{minipage}{.17\linewidth}
\includegraphics[scale=0.22]{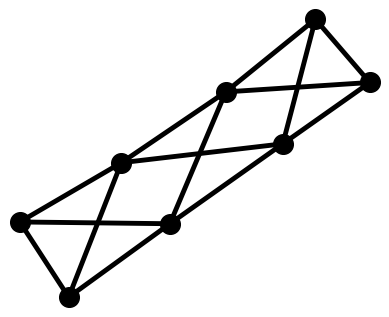}
\end{minipage}
\hspace{.01\linewidth}

\caption{These and their complements are the $8$-vertex forbidden induced subgraphs for apex cographs.}

\label{eight}

\end{figure}

{\bf Graphs on nine vertices.} The algorithm found no 9-vertex forbidden induced subgraphs.  However, we felt it unsatisfactory to rely on a computer program for such a significant fact, so we produced a readable proof by analyzing a possible $9$-vertex forbidden induced subgraph for apex cographs, and ruling out that possibility. That proof occupies the rest of the section.  Henceforth, $G$ denotes a $9$-vertex forbidden induced subgraph for apex cographs.

\begin{lemma}
\label{9_vtx_one}
$V(G)$ has three distinct subsets $S,T,$ and $F$ such that each of $G[S]$, $G[T],$ and $G[F]$ induces a path $P_4$, and $|S \cap T| = |T \cap F| = |S \cap F| =1$.    
\end{lemma}

\begin{proof}
By Proposition \ref{bound_k}, $G$ has distinct subsets $S$ and $T$ such that both $G[S]$ and $G[T]$ induce the path $P_4$, and $|S \cap T| = 1$. Let $S \cap T = \{\alpha \}$. Since $G - \alpha$ is not a cograph, there is a subset $F$ of $V(G) - \alpha$ such that $G[F] \cong P_4$. Since $V(G) = S \cup T \cup F$, we have $|T \cap F| = |S \cap F| = 1$.
\end{proof}

Note that $|S \cap T \cap F| = 0$, otherwise $|V(G)| > 9$. 

In the rest of the section we let $S = \{\alpha, s_1, s_2, \beta\}$, $T = \{\alpha, t_1, t_2, \gamma\},$ and $F = \{\beta, f_1, f_2, \gamma\}$ denote the subsets of $V(G)$  as in Lemma \ref{9_vtx_one}. Thus, $V(G) = \{\alpha, \beta, \gamma, s_1, s_2, t_1, t_2, f_1, f_2\}$.

\begin{lemma}
\label{9_vtx_two}
Let $\{x,y\} = \{\alpha, f_1\}, \{\alpha, f_2\}$, $\{\beta, t_1\}, \{\beta, t_2\}, \{\gamma, s_1\}$, or $\{\gamma, s_2\}$. Then $G - \{x,y\}$ is a cograph.
\end{lemma}

\begin{proof}
For $i=1,2$, the graph $G-f_i$ is an apex cograph, so it has a vertex $z$ in $V(G)-f_i$ such that $G-\{f_i, z\}$ is a cograph. We must have $z = \alpha$; otherwise $V(G) - \{f_i,z\}$ would contain $S$ or $T$, thus an induced $P_4$, a contradiction. Replacing $f_i$ with $t_i$, our result follows for $\{\beta, t_i\}$, and similarly for $\{\gamma, s_i\}$.
\end{proof}

For two $P_4$'s that have one common vertex, there are three distinct cases based on how they intersect with each other. The cases are shown in Figure \ref{L_X_T}; we call them the $L$-case, $T$-case, and $X$-case as in the figure.

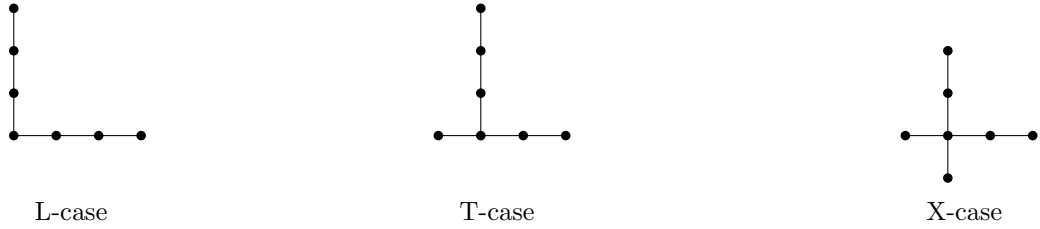
\begin{figure}[htbp]
\centering

\begin{tikzpicture}[scale=0.4,colorstyle/.style={circle, draw=black!100,fill=black!100, thick, inner sep=1.25pt, minimum size=0.5mm}]

    \node (1) at (-25,1)[colorstyle]{};
    \node (2) at (-25,2.5)[colorstyle]{};
    \node (3) at (-25,4)[colorstyle]{};
    \node (4) at (-25,5.5)[colorstyle]{};

    \node (b) at (-23.5,1)[colorstyle]{};
    \node (c) at (-22,1)[colorstyle]{};
    \node (d) at (-20.5,1)[colorstyle]{};

    \draw[thick] (1)--(2)--(3)--(4);
    \draw[thick] (1)--(b)--(c)--(d);

    \node (5) at (-14,1)[colorstyle]{};
    \node (6) at (-14,2.5)[colorstyle]{};
    \node (7) at (-14,4)[colorstyle]{};
    \node (8) at (-14,5.5)[colorstyle]{};

    \node (e) at (-15.5,1)[colorstyle]{};
    \node (f) at (-12.5,1)[colorstyle]{};
    \node (g) at (-11,1)[colorstyle]{};

    \draw[thick] (5)--(6)--(7)--(8);
    \draw[thick] (e)--(5)--(f)--(g);

   \node (9) at (-4,1)[colorstyle]{};
    \node (10) at (-4,2.5)[colorstyle]{};
    \node (11) at (-4,4)[colorstyle]{};
    \node (12) at (-4,5.5)[colorstyle]{};

    \node (h) at (-5.5,2.5)[colorstyle]{};
    \node (i) at (-2.5,2.5)[colorstyle]{};
    \node (j) at (-1,2.5)[colorstyle]{};

    \draw[thick] (9)--(10)--(11)--(12);
    \draw[thick] (h)--(10)--(i)--(j);

\node (n1) at (-23,-0.5){L-case};

\node (n1) at (-13.5,-0.5){T-case};

\node (n1) at (-4,-0.5){X-case};

\end{tikzpicture}

\caption{The intersection cases of two $P_4$'s.}
\label{L_X_T}

\end{figure}

\begin{lemma}
\label{9_vtx_three}
Let $S$ and $T$ be distinct subsets of $V(G)$ as in Lemma \ref{9_vtx_one}. If the intersection of $G[S]$ and $G[T]$ is an $L$-case, then for $i=1,2$ the graph $G-\{\alpha, f_i\}$ is connected.
\end{lemma}

\begin{proof}
We denote the intersection of $G[S]$ and $G[T]$ as shown in Figure \ref{S_T_intersect}, where $\{1,2,3\} = \{s_1,s_2, \beta\}$ and $\{4,5,6\} = \{t_1,t_2,\gamma\}$. First suppose that there is no edge between $\{1,2,3\}$ and $\{4,5,6\}$ in $G$, so $1\alpha45$ and $4\alpha12$ are $P_4$'s in $G$. By Lemma \ref{9_vtx_two}, it follows that $1=\beta$ or $6=\gamma$, and $4= \gamma$ or $3=\beta$. Note that if $6=\gamma$, then $1\alpha45$ is a $P_4$ in $G- \{\gamma, s\}$ for $s=s_1$ or $s_2$, a contradiction to Lemma \ref{9_vtx_two}. Similarly, if $3=\beta$, then $4\alpha12$ is a $P_4$ in $G-\{\beta,t\}$ for $t=t_1$ or $t_2$ contradicting Lemma \ref{9_vtx_two}. Therefore $1=\beta$ and $4=\gamma$. 

For $\{i,j\}=\{1,2\}$, observe that if $f_j$ has neighbours in both $\{1, 2, 3\}$ and $\{4,5,6\}$ in $G$, then $G-\{\alpha,f_i\}$ is connected and our result follows. Since $\gamma \beta$ is not an edge in $G$ and $f_j$ has exactly one neighbour in $\{\beta, \gamma\}$ , we may assume that the $P_4$ induced on $\{\beta, \gamma, f_i, f_j\}$ is $\beta f_if_j \gamma$ or $\beta f_i \gamma f_j$. 
If $\beta f_if_j \gamma$ is a $P_4$ in $G$, then $f_i2$ is an edge in $G$; otherwise $f_j f_i \beta 2$ would be an induced $P_4$ in $G - \{\gamma, 3\}$.  Since $1=\beta$, $3$ is either $s_1$ or $s_2$, which is a contradiction to Lemma \ref{9_vtx_two}. Also $f_i3$ is an edge in $G$, else $f_j f_i 23$ is a $P_4$ in $G - \{\beta, t\}$ for $t=t_1$ or $t_2$. It follows that $\gamma f_j f_i 3$ is a $P_4$ in $G-\{\beta, t\}$ for $t=t_1$ or $t_2$, a contradiction to Lemma \ref{9_vtx_two}. Therefore $\beta f_i \gamma f_j$ is a $P_4$ in $G$. 
Observe that $f_i2$ is an edge in $G$, else $\gamma f_i \beta 2$ is a $P_4$ in $G-\{\alpha, f_j\}$, a contradiction to Lemma \ref{9_vtx_two}. Also $f_i3$ is an edge in $G$, else $\gamma f_i 23$ is a $P_4$ in $G-\{\alpha, f_j\}$, again a contradiction to Lemma \ref{9_vtx_two}. It follows that $f_j \gamma f_i 3$ is a $P_4$ in $G-\{\beta, t\}$ for $t=t_1$ or $t_2$, a contradiction. Therefore, there is an edge between $\{1,2,3\}$ and $\{4,5,6\}$. 

Note that $f_j$ has no neighbour in $\{\beta,\gamma,s_1,s_2,t_1,t_2\}$ in $G$; otherwise $G-\{\alpha, f_i\}$ would be connected. Therefore we may assume that the $P_4$ induced by $\{f_i,f_j,\gamma,\beta\}$ is $f_jf_i\beta\gamma$. Observe that, for a neighbour $x$ of $f_i$ in $\{1,2,3\}$, a neighbour $v$ of $x$ in $\{1,2,3\}$ is also a neighbour of $f_i$ in $G$; otherwise $f_jf_ixv$ is a $P_4$ in $G-\{\gamma, s\}$ or $G - \{\beta, t\}$ where $s$ is in $\{s_1,s_2\}$ and $t$ is in $\{t_1,t_2\}$, a contradiction to Lemma \ref{9_vtx_two}. It follows that $f_i$ is a neighbour of every vertex in $\{1,2,3\}$ in $G$. 
Note that $f_j \alpha$ is not an edge in $G$; otherwise $f_j\alpha 45$ would be a $P_4$ in $G-\{\gamma, s\}$ or $G - \{\beta, t\}$ where $s$ is in $\{s_1,s_2\}$ and $t$ is in $\{t_1,t_2\}$, a contradiction to Lemma \ref{9_vtx_two}. Observe that if $f_i\alpha$ is not an edge in $G$, then $f_jf_i1\alpha$ is a $P_4$ in $G-\{\gamma, s\}$ for $s=s_1$ or $s_2$, contradicting Lemma \ref{9_vtx_two}. Therefore $f_i \alpha$ is an edge in $G$. If $f_i4$ is not an edge in $G$, then $f_jf_i\alpha 4$ is a $P_4$ in $G - \{\beta,t\}$ for $t=t_1$ or $t_2$, a contradiction. 
Observe that $f_i5$ is also an edge in $G$; otherwise $f_jf_i45$ is a $P_4$ in $G-\{\gamma, s\}$ or $G - \{\beta, t\}$ where $s$ is in $\{s_1,s_2\}$ and $t$ is in $\{t_1,t_2\}$, a contradiction. Since $f_i\gamma$ is not an edge in $G$, it follows that $6=\gamma$, so $f_jf_i56$ is a $P_4$ in $G-\{\beta, t\}$ for $t=t_1$ or $t_2$, a contradiction to Lemma \ref{9_vtx_two}. Therefore our result holds.
\end{proof}

\begin{figure}[htbp]
\centering

\includegraphics{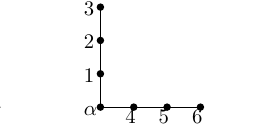}

\caption{Intersection of $G[S]$ and $G[T]$.}
\label{S_T_intersect}

\end{figure}

\begin{lemma}
\label{9_vtx_four}
Let $S$, $T,$ and $F$ be distinct subsets of $V(G)$ as in Lemma \ref{9_vtx_one}. Then there is at least one pair in $\{G[S], G[T], G[F]\}$ whose intersection is not a $T$-case.  
\end{lemma}

\begin{proof}
Assume to the contrary that the intersection of every pair in $\{G[S], G[T], G[F]\}$ is a $T$-case. Then we have six distinct cases depending upon how $G[S]$, $G[T],$ and $G[F]$ intersect with each other in $G$. The cases are shown in Figure \ref{TTT}. Observe that the cases $1$, $2,$ and $3$ are complementary to cases $4$, $5,$ and $6$, respectively. By Lemma \ref{closed_complement}, it is enough to consider case 1, case 2, and case 3. 

Suppose that the intersection of $G[S]$, $G[T],$ and $G[F]$ is case 1; then the intersection of $\overline{G}[S]$, $\overline{G}[T],$ and $\overline{G}[F]$ is case 4. Observe that if  there is no edge between $t_2$ and $\{s_1,s_2\}$ in $G$, then $t_2\alpha s_1 s_2$ is an induced $P_4$ in $G -\{\beta,t_1\}$, a contradiction to Lemma \ref{9_vtx_two}. Similarly, if there is no edge between $\{s_1,s_2\}$ and $\{\gamma,t_1\}$, then $\gamma \alpha s_1 s_2$ is an induced $P_4$ in $G-\{\beta,t_1\}$, a contradiction. Therefore $G-\{\alpha, f_2\}$ is connected. 
If there is no edge between $s_2$ and $\{t_1,t_2\}$ in $\overline{G}$, then $s_2\alpha t_1 t_2$ is an induced $P_4$ in $\overline{G} - \{\gamma, s_1\}$, a contradiction to Lemma \ref{9_vtx_two}. Therefore $\overline{G}-\{\alpha,f_2\}$ is connected. This contradicts Lemma \ref{9_vtx_two} since $G-\{\alpha,f_2\}$ is a cograph. 

Next suppose that the intersection of $G[S]$, $G[T],$ and $G[F]$ is case 2. Observe that if there is no edge between $s_2$ and $\{t_2, \gamma\}$ in $G$, then $s_2\alpha t_2 \gamma$ is an induced $P_4$ in $G-\{\beta, t_1\}$, a contradiction to Lemma \ref{9_vtx_two}. Similarly, there is an edge between $s_1$ and $\{t_2,\gamma\}$ in $G$, so $G-\{\alpha,f_2\}$ is connected. Observe that if there is no edge between $t_1$ and $\{\beta, s_1\}$ in $\overline{G}$, then $t_1 \alpha \beta s_1$ is an induced $P_4$ in $G-\{\gamma, s_2\}$, a contradiction to Lemma \ref{9_vtx_two}. Therefore $\overline{G}-\{\alpha, f_2\}$ is connected, a contradiction.

Finally suppose that the intersection of $G[S]$, $G[T],$ and $G[F]$ is case 3. If there is no edge between $s_1$ and $\{t_2,\gamma\}$ in $G$, then $s_1\alpha t_2 \gamma$ is an induced $P_4$ in $G-\{\beta,t_1\}$, a contradiction to Lemma \ref{9_vtx_two}. It follows that $G-\{\alpha,f_1\}$ is connected. If there is no edge between $t_1$ and $\{\beta,s_1\}$ in $\overline{G}$, then $t_1 \alpha \beta s_1$ is an induced $P_4$ in $G-\{\gamma, s_2\}$, a contradiction. By Lemma \ref{9_vtx_two}, $\overline{G} - \{\alpha, f_1\}$ is disconnected, so we may assume that there is no edge between $f_2$ and $\{t_1, t_2\}$ in $\overline{G}$. It follows that either $\alpha f_2$ is not an edge in $\overline{G}$ and $f_2 \gamma \alpha t_1$ is an induced $P_4$ in $\overline{G}$, or $\alpha f_2$ is an edge in $\overline{G}$ and $f_2 \alpha t_1 t_2 $ is an induced $P_4$ in $\overline{G}$. Then $f_2 \gamma \alpha t_1$ is an induced $P_4$ in $\overline{G} - \{\beta, t_2\}$, a contradiction to Lemma \ref{9_vtx_two}, and $f_2 \alpha t_1 t_2 $ is an induced $P_4$ in $\overline{G} - \{\gamma, s_1\}$, again a contradiction to Lemma \ref{9_vtx_two}. Therefore $\overline{G} - \{\alpha, f_1\}$ is connected. This is a contradiction, so our result holds.
\end{proof}

\begin{figure}[htbp]
\centering

\hspace*{-1.1cm}\includegraphics[scale=0.25]{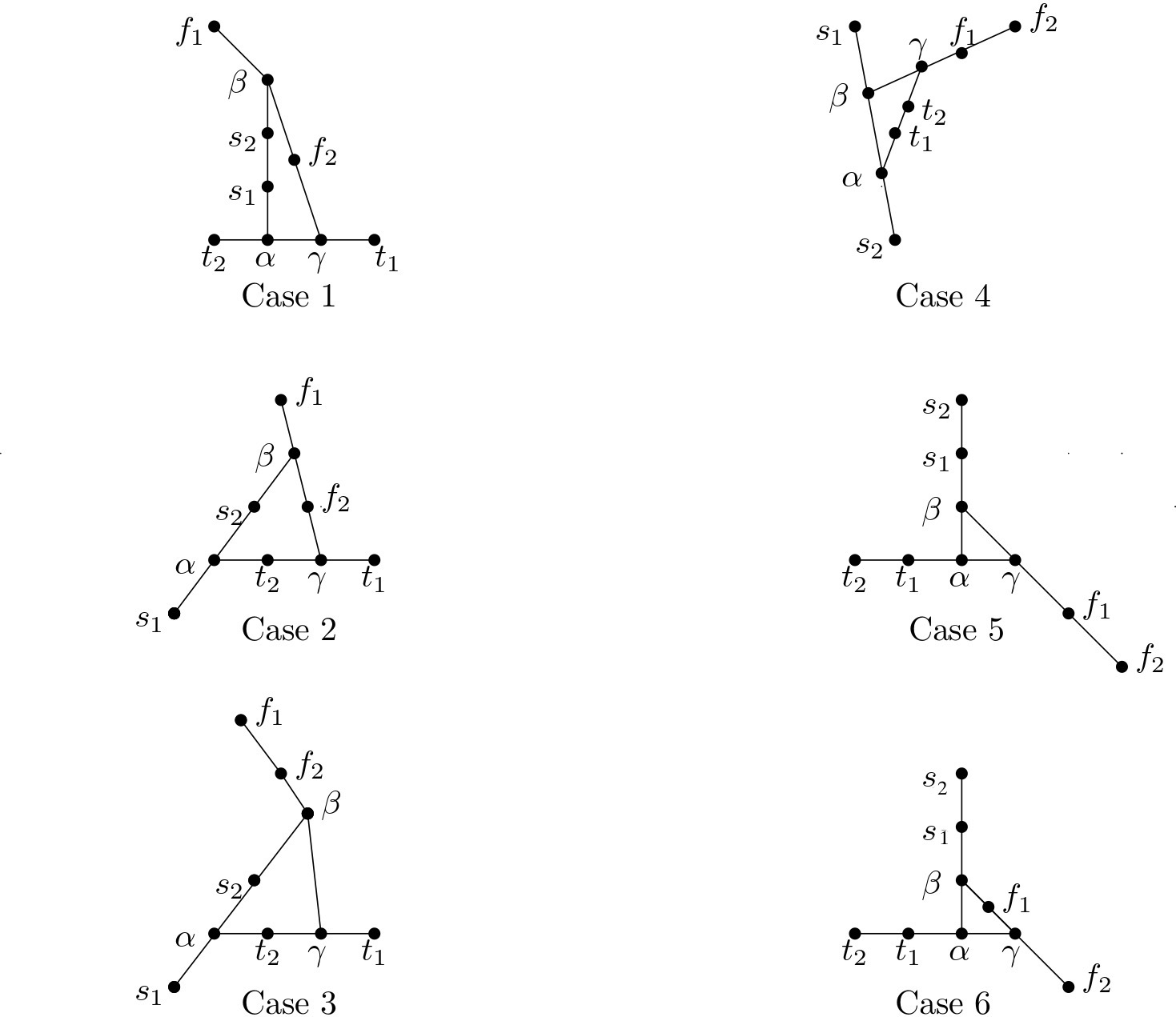}

\caption{The distinct cases where all intersections are $T$-cases.}
\label{TTT}

\end{figure}

\begin{proposition}
\label{main_9_vtx}
There is no $9$-vertex forbidden induced subgraph for apex cographs.
\end{proposition}

\begin{proof}
Assume that $G$ is a $9$-vertex forbidden induced subgraph for apex cographs. Let $S,T,$ and $F$ be distinct subsets of $V(G)$ as in Lemma \ref{9_vtx_one}. By Lemma \ref{closed_complement} and Lemma \ref{9_vtx_four}, we may assume that the intersection of $G[S]$ and $G[T]$ is an $L$-case, as shown in Figure \ref{L_case_G[S,T]}. The intersection of $\overline{G}[S]$ and $\overline{G}[T]$ is an $X$-case. By Lemma \ref{9_vtx_three}, for $i$ in $\{1,2\}$ the graph $G-\{\alpha, f_i\}$ is connected, so by Lemma \ref{9_vtx_two}, the complement $\overline{G}-\{\alpha, f_i\}$ is disconnected. 

Suppose that there is no edge between $\{1,2\}$ and $\{5,6\}$ in $\overline{G}$. Then $2\alpha 56$ and $5\alpha 21$ are $P_4$'s in $\overline{G}$. Observe that if $2 \neq \beta$ then $4 = \gamma$ and $2\alpha 56$ is an induced $P_4$ in $\overline{G}-\{\gamma, s\}$ for $s=s_1$ or $s_2$, a contradiction to Lemma \ref{9_vtx_two}. Therefore $2 = \beta$. By symmetry, $5=\gamma$. 
If there is no edge between $\{3\}$ and $\{\gamma, 6\}$ in $\overline{G}$, then $3\alpha \gamma 6$ is an induced $P_4$ in $\overline{G} - \{\beta, t\}$ for $t=t_1$ or $t_2$, a contradiction to Lemma \ref{9_vtx_two}. Therefore there is an edge between $\{3\}$ and $\{\gamma, 6\}$ in $\overline{G}$. Similarly, there is an edge between $\{4\}$ and $\{1, \beta\}$ in $\overline{G}$. 
For $\{i,j\} = \{1,2\}$, if $f_i$ has neighbours in $\{1,\beta\}$ and $\{\gamma,6\}$ in $\overline{G}$, then $\overline{G} - \{\alpha, f_j\}$ is connected, a contradiction. Therefore $f_i$ has no neighbour in $\{1, \beta\}$ or $\{\gamma, 6\}$ in $\overline{G}$. Similarly, $f_j$ has no neighbour in $\{1, \beta\}$ or $\{\gamma, 6\}$ in $\overline{G}$. Since $\{\beta, \gamma, f_1, f_2\}$ induces an induced $P_4$, we may assume that the $P_4$ is $\gamma f_i f_j \beta$. Note that $f_j$ has no neighbour in $\{3, \gamma, 6\}$ in $\overline{G}$, and $f_i$ has no neighbour in $\{1,\beta,4\}$ in $\overline{G}$. It follows that either $4\alpha \gamma f_i$ is an induced $P_4$ in $\overline{G}$, or $\alpha f_i$ is an edge in $\overline{G}$. Thus, $4 \alpha \gamma f_i$ is an induced $P_4$ in $\overline{G}-\{\beta,t\}$ for $t=t_1$ or $t_2$. This is a contradiction to Lemma \ref{9_vtx_two} so we may assume that $\alpha f_i$ is an edge. 
It follows that $1 \beta \alpha f_i$ is an induced $P_4$ in $\overline{G} - \{\gamma, s\}$ for $s=s_1$ or $s_2$, a contradiction. Therefore we may assume that there is an edge between $\{1,2\}$ and $\{5,6\}$ in $\overline{G}$.

Suppose that there is no edge between $\{3\}$ and $\{5,6\}$ in $\overline{G}$. Then $3 \alpha 56$ is an induced $P_4$, so by Lemma \ref{9_vtx_two} we have $3=\beta$ or $4 = \gamma$.  If $4=\gamma$, then $3\alpha 56$ is an induced $P_4$ in $\overline{G}-\{\gamma, s\}$ for $s=s_1$ or $s_2$, a contradiction to Lemma \ref{9_vtx_two}. Therefore $3=\beta$ and $4 \neq \gamma$. If there is no edge between $\{4\}$ and $\{1,2\}$ in $\overline{G}$, then $4\alpha 21$ is an induced $P_4$ in $\overline{G}-\{\beta,t\}$ where $t$ is in $\{t_1,t_2\}$, a contradiction to Lemma \ref{9_vtx_two}. Therefore there is an edge between $\{4\}$ and $\{1,2\}$ in $\overline{G}$. 
Since $\gamma \in \{5,6\}$ and there is no edge between $\{3\}$ and $\{5,6\}$ in $\overline{G}$, it follows that $\gamma \beta$ is not an edge in $\overline{G}$. 
Observe that $\{\gamma, \beta, f_i, f_j\}$ induces an induced $P_4$ in $\overline{G}$ so both $\gamma$ and $\beta$ have a neighbour in $\{f_i, f_j\}$ in $\overline{G}$. If both $\gamma$ and $\beta$ are neighbours of $f_i$ in $\overline{G}$, then $\overline{G}-\{\alpha, f_j\}$ is connected, a contradiction. 

Since $\beta \gamma$ is not an edge in $\overline{G}$, it follows that the $P_4$ induced by $\{\gamma, \beta, f_i, f_j\}$ in $\overline{G}$ is $\beta f_i f_j \gamma$ or $\beta f_j f_i \gamma$. In the former case $f_i$ has no neighbour in $\{5,6\}$ in $\overline{G}$, else $\overline{G}-\{\alpha, f_j\}$ would be connected, a contradiction. 
It follows that $f_i\beta \alpha 5$ is an induced $P_4$ unless $\alpha f_i$ is an edge in $\overline{G}$. 
Now $f_i\beta \alpha 5$ is an induced $P_4$ in $\overline{G}-\{\gamma,s\}$ for $s=s_1$ or $s_2$. This contradicts Lemma \ref{9_vtx_two}, so $\alpha f_i$ is an edge in $\overline{G}$. 
It follows that $f_i \alpha 56$ is an induced $P_4$ in $\overline{G}-\{\beta,t\}$ for $t$ in $\{t_1, t_2\}$, a contradiction to Lemma \ref{9_vtx_two}. Therefore in the former case there is an edge between $\{3\}$ and $\{5,6\}$ in $\overline{G}$. The latter case is similar. 

Similarly, there is an edge between $\{4\}$ and $\{1,2\}$ in $\overline{G}$.

It now follows that $\overline{G}-\{\alpha,f\}$ is connected for $f$ in $\{f_1,f_2\}$, a contradiction. This completes the proof that there is no 9-vertex forbidden induced subgraph for apex cographs.
\end{proof}

\begin{figure}[htbp]
\centering

\includegraphics{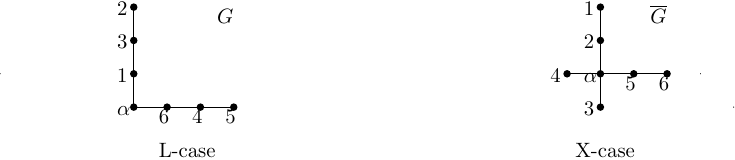}

\caption{$L$-case intersection of $G[S]$ and $G[T]$, where $S= \{1,2,3\}$ and $T =\{4,5,6\}$.}
\label{L_case_G[S,T]}

\end{figure}

\section*{Declarations}

V. Sivaraman was partially supported by Simons Foundation Travel Support for Mathematicians (Award No. 855469).

\end{document}